\documentclass[12pt,leqno]{amsart}

\usepackage{pifont}
\usepackage{stmaryrd}
\usepackage{dsfont}
\usepackage{mathrsfs}
\usepackage{amsmath}
\usepackage{amssymb}
\usepackage{hyperref}
\usepackage{bookmark}
\usepackage[bottom]{footmisc}
\usepackage{verbatim}
\usepackage{color, xcolor}
 \usepackage[all]{xy}
 \usepackage{amsmath, mathrsfs}
 \usepackage{extarrows}

\def\Z{\mathbb{Z}}

\def\C{\mathbb{C}}

\numberwithin{equation}{section}
\newtheorem{theo}{Theorem}[section]

\newtheorem{lemm}[theo]{Lemma}
\newtheorem{prop}[theo]{Proposition}

\allowdisplaybreaks

\begin{document}

\title[Derivations of rational VOAs are inner]{Derivations of rational vertex operator algebras are inner}

\author{Jianzhi Han}

\address{School of Mathematical Sciences, Key Laboratory of Intelligent Computing and Applications (Ministry of Education), Tongji University, Shanghai, 200092,
China.}\email{jzhan@tongji.edu.cn}




\begin{abstract}
We show that every derivation of a simple and rational vertex operator algebra of  CFT type is an inner derivation.
\end{abstract}

\maketitle

\section{Induction}

Lie algebra theory lays a fundamental foundation for the study of algebraic structures and symmetry representations in mathematics and theoretical physics. A classical and pivotal result in finite-dimensional Lie algebra theory states that the finite-dimensional module category of a semisimple Lie algebra is completely reducible, which is commonly known as Weyl's semisimplicity theorem (see \cite{W1,W2,W3}). Equally essential is the classical conclusion that all derivations of semisimple  Lie algebras are inner  (cf.\,\cite{Hum}), which reveals the rigidity of the derivation structure of semisimple Lie algebras and plays an important role in the investigation of automorphism groups, deformation theories and symmetry properties of Lie algebras.

In vertex operator algebra (VOA) theory, the notion of rationality serves as the precise analogue of semisimplicity for finite-dimensional Lie algebras. By definition, a vertex operator algebra is rational if its category of admissible modules is semisimple, which implies the complete reducibility of all admissible representations and endows the VOA with excellent structural and representation-theoretic properties (see \cite{DLM,Z}). Rational vertex operator algebras constitute the most important and well-behaved class of vertex operator algebras, covering almost all classical VOA models closely related to physical backgrounds. Parallel to the classical Lie algebra problem of characterizing derivations, it is a fundamental  problem in VOA theory to clarify the structural relationship between the full derivation space and the inner derivation subspace of a rational vertex operator algebra. In fact, it has been conjectured that these two spaces coincide.

The conjecture is supported by two key pieces of evidence. First, the conjecture holds for all fundamental classes of rational vertex operator algebras, including the moonshine module $V^\natural$ \cite{B,D2,FLM}, the Virasoro vertex operator algebra $L(c,0)$ \cite{DMZ,FZ,W}, the affine vertex operator algebra $L_{\hat{\mathfrak g}}(l,0)$ \cite{DL,FZ,L}, and the lattice vertex operator algebra $V_L$ \cite{B,D1,FLM}. Second, every derivation of a rational vertex operator algebra is locally inner. Precisely, let $V=\bigoplus_{i\in\mathbb Z_+}V_i$ be a rational vertex operator algebra and let $d$ be any derivation of $V.$ For each nonnegative integer $n$, there exists some $u^{(n)}\in V$ such that $d=o(u^{(n)})$ on $\bigoplus_{i=0}^nV_i$. Here, $o$ denotes the linear map from $V$ to $\operatorname{End}V$ that sends each homogeneous vector $v$ to its zero mode operator $v_{\operatorname{wt}v-1}$. This local inner property follows from the existence of a surjective associative algebra homomorphism from $A_n(V)$ \cite{DLM3,Z} to $\bigoplus_{i=0}^n \operatorname{End}V_i$.

The Killing form, a canonical nondegenerate invariant bilinear form, serves as a fundamental tool for proving the inner property of all derivations on semisimple Lie algebras. However, such a Killing form cannot be defined for infinite-dimensional vertex operator algebras, and this absence constitutes a major obstacle to establishing that all derivations of a rational vertex operator algebra are inner.   The first attempt to relate derivations and inner derivations by means of  nondegenerate bilinear forms was made in \cite{DG}. In that work, the authors decomposed the full derivation space of a vertex operator algebra into the direct sum of the inner derivation subspace and its orthogonal complement with respect to invariant bilinear forms. Nevertheless, the triviality of the orthogonal complement, which is the key to establishing the coincidence of inner and general derivations, remained unproven and thus left the conjecture open.

To overcome the difficulties caused by the lack of Killing-type forms and the gap between local and global inner properties, we develop an analytical argument centered on the local approximation property of derivations and limit analysis in this paper. Although the problem investigated is purely algebraic in nature, our proof  introduces analytic limit techniques to bridge the gap between local finite-dimensional approximations and global infinite-dimensional structural identities. The core idea of our approach is as follows. For any derivation $d$ of a rational VOA $V,$ the locally inner property yields a sequence of  elements $u^{(n)}$ that realize the derivation $d$ on finite graded truncations. Due to the infinite-dimensional graded structure of $V,$ the finite-level matching identities $d=o(u^{(n)})$ on $\bigoplus_{i=0}^n V_i$ do not guarantee a uniform global element in $V,$ and the pointwise limit of the sequence $\{u^{(n)}\}$ may lie outside the original VOA space. To resolve this issue, we extend the  space  $V$ to its complete direct product space $\prod_{i\in\mathbb Z_+}V_i$, which accommodates all limit elements of graded sequences.

We first prove that for every derivation $d$, there exists an $o$-additive element $u\in \prod_{i\in\mathbb Z_+}V_i$ such that the global identity $d=o(u)$ holds on the entire space $V.$ We then introduce the key set $E_d$ consisting of all such $o$-additive representative elements for $d$ in the direct product space, and establish a precise characterization: the derivation $d$ is inner if and only if $E_d$ contains an element with an infinite $L(1)$-kernel index. Combined with the finite generation property of rational vertex operator algebras and the continuity of zero-mode operators, we further eliminate some components of representative elements in $E_d$ via iterative reduction and limit arguments. Finally, we extract a valid element in the weight-one subspace $V_1$ that globally realizes the derivation $d$, which completely verifies the inner property of arbitrary derivations.

We assume familiarity with fundamental notions of vertex operator algebra theory, including the definition of vertex operator algebras (cf.\,\cite{B,DL,FHL,FLM,LL}) and the definition of admissible  modules (see \cite{Z}).

\section{Derivation}

Let $V=\bigoplus_{i\in\mathbb Z}V_i$ be a vertex operator algebra and $M$ an admissible $V$-module. An element $u\in V$ is said to be homogeneous if $u\in V_n$ for some $n\in\mathbb Z$, in which case we set $\operatorname{wt}u = n$.
We define a linear map $o\colon V\to\operatorname{End}\,M$ which sends each homogeneous vector $v\in V$ to $v_{\operatorname{wt}v-1}$.
Let $\prod_{i\in\mathbb Z}V_i$ denote the direct product of the graded components $V_i.$ An element $\sum_{i\in\mathbb Z}v^i\in\prod_{i\in\mathbb Z}V_i$ with $v^i\in V_i$ is called \textit{$o$-additive} on $V$ provided that for every $u\in V,$ there exists $n\in\mathbb Z_+$ satisfying $o(v^i)u=0$ for all $i>n$; equivalently, the sum $\sum_{i\in\mathbb Z}o(v^i)u$ is finite.
Evidently, every element of $V$ is $o$-additive on $V.$ Since $o(L(-1)u+L(0)u)=0$ holds for all $u\in V,$ we obtain
\begin{equation}\label{lift-de}
o(L(-1)v)=-\operatorname{wt}v\,o(v).
\end{equation}
Closely related to \eqref{lift-de} is the so-called radical \(J(V)\) of $V,$ which denotes the subspace consisting of all \(u\in V\) satisfying \(o(u)=0\). It was shown in \cite{DLMM} that
 \begin{equation}\label{j(v)} J(V)=J(V)\cap V_1+(L(0)+L(-1))V.\end{equation}

 Recall from \cite{DG} that a {\em derivation} $d$ of $(V,Y, {\bf1},\omega)$ is an endomorphism of $V$ such that $d{\bf 1}=d\omega=0$ and $[d,u_n]=(du)_n$ for all  $u\in V$ and $n\in\Z$. In particular, \begin{equation}\label{d=comm}[d, L(n)]=0 \quad \text{for all}\  n\in\Z,\end{equation} where $L(n)=\omega_{n+1}$. Consequently, every derivation preserves all homogenous subspaces $V_i$. A derivation $d$ is called an {\em inner derivation} if $d=o(v)$ for some $v\in V.$ Denote by $\operatorname{Der}V$ and $\operatorname{IDer}V$ the set of derivations of $V$ and the set of inner derivations of $V,$ respectively. It should be pointed out that derivations of a vertex operator algebra can also be characterized in terms of its first cohomology \cite{Hua}. A vertex operator algebra $V$ is called to be of  {\em CFT type} if $V=\bigoplus_{i\in\Z_+}V_i$ and  $V_0=\C{\bf1}$.  
Then we have:

\begin{prop}\label{cit-fro-dg}\cite{DG} Let $V$ be a vertex operator algebra of  CFT type. Then
$\operatorname{IDer}\,V=\{o(v)\mid v\in V_1\}.$
\end{prop}

Associated to $V$ is a family of associative algebras $A_n(V)=V/O_n(V)$ indexed by $n\in\mathbb{Z}_+$, constructed in \cite{DLM3,Z}. For $A_n(V)$, we recall the following result.

\begin{theo} \cite {DLM3}\label{cite-An(V)}
Suppose that $V$ is a   rational  vertex operator algebra.   Then
  $A_n(V)$ is finite dimensional and the linear map $v\mapsto o(v)$ induces an algebra isomorphism
\begin{equation}\label{asso-iso}A_n(V)\cong \bigoplus_{i=1}^s \bigoplus_{l=0}^{n}\operatorname{End} W^i(l),\end{equation} where $W^i=\bigoplus_{n\in\Z_+}W^i(n)$ with $W^i(0)\neq0$ for $i=1,2,\ldots, s$  are all the inequivalent irreducible modules of $V$.
\end{theo}

A subset  $S$ of $V$ is called a {\em generating set} of $V$ if \begin{equation}\label{spanning}V={\rm span}\{v_{n_1}^{(1)}\cdots v^{(r)}_{n_r}{\bf 1}\mid r\in \Z_+, v^{(1)},\ldots, v^{(r)}\in S, n_1,\ldots, n_r\in\Z\}.\end{equation}  And $V$ is called to be  {\em  finitely generated} if $V$ has a finite generating set.

\begin{lemm}\label{lem-m}Let $V$ be a vertex operator algebra and $d_1,d_2\in \operatorname{Der}V.$
Suppose that $S$ is a generating set of  $V.$ If $d_1=d_2$ on $S,$ then $d_1=d_2$.
\end{lemm}
\begin{proof}
Let $d$ be a derivation of $V$  such that $d|_S=0$. It is sufficient to show $d=0$. But this follows immediately from \eqref{spanning} and the definition of a derivation.
\end{proof}

The lemma below constitutes a key step toward determining the derivations of a rational vertex operator algebra.
\begin{lemm}\label{existence} Let $V$ be a  simple rational vertex operator algebra of CFT type.
For any $d\in \operatorname {Der}V,$ there exists $v=\sum_{i=1}^\infty v^i\in \prod_{i=1}^\infty V_i$ with $v^i\in V_i$ for all $i$ such that $v$ is $o$-additive and $d=\sum_{i=1}^\infty o(v^i)$.
\end{lemm}

\begin{proof}
By Theorem \ref{cite-An(V)}, each $A_n(V)$ is finite-dimensional. Hence for every $n\in\mathbb Z_+$, we may pick the minimal integer $N_n\in\mathbb Z_+$ satisfying
\[
V=\bigoplus_{i=0}^{N_n}V_i + O_n(V).
\]
The resulting sequence $\{N_n\}_{n\ge 1}$ is strictly increasing. Now fix a sufficiently large integer $t$. Then by \eqref{asso-iso}, there exist $u^{(t+i)}\in \bigoplus_{k=0}^{N_{t+i}}V_k$ for $i\geq0$ such that \begin{equation}\label{un}d|_{\bigoplus_{n=0}^t V_n}=o(u^{(t)})|_{\bigoplus_{n=0}^t V_n}\quad {\rm and}\quad o(u^{(t+i+1)})|_{\bigoplus_{n=0}^{t+i} V_n}=0 \end{equation} and  that \begin{equation}\label{sumadd} o(u^{(t)}+u^{(t+1)}+\cdots+u^{(t+i+1)})|_{V_{t+i+1}}=d|_{V_{t+i+1}}.\end{equation}  Write $u^{(t+i)}=\sum_{k=0}^{N_{t+i}}u^{(t+i,k)}$ with $u^{(t+i,k)}\in V_k$ for  all $0\le k\le N_{t+i}$.
 Since $V_0=\mathbb C\mathbf 1$, we have $u^{(t+i,0)}=\lambda_i\mathbf 1$ for some $\lambda_i\in\mathbb C$. From \eqref{un},
\[
0=o(u^{(t+i)})\mathbf 1=o(u^{(t+i,0)})\mathbf 1=\lambda_i\mathbf 1,
\]
which forces $u^{(t+i)}\in \bigoplus_{k=1}^{N_{t+i}}V_k$. Furthermore, in view of \eqref{lift-de}, we may substitute $u^{(t+i+1)}$ with an element $v^{(t+i+1)}\in V_{N_{t+i+1}}$ satisfying $o(u^{(t+i+1)})=o(v^{(t+i+1)}).$ So we may assume
\[
u^{(t+i+1)}\in V_{N_{t+i+1}}\quad \text{for\ all}\ i\geq0.
\]
Set $u=\sum_{i=0}^\infty u^{(t+i)}\in \prod_{i=1}^\infty V_i$.  Now by \eqref{un} and \eqref{sumadd}, it is straightforward to verify that $u$ is $o$-additive and $d=\sum_{i=0}^\infty o(u^{(t+i)})$, completing the proof.
\end{proof}

For each $d\in\operatorname{Der}V,$ define
\[
E_d=\biggl\{v=\sum_{i=1}^\infty v^{(i)}\in \prod_{i=1}^\infty V_i\ \text{with}\ v^{(i)}\in V_i\;\bigg|\;v\text{ is }o\text{-additive and }d=\sum_{i=1}^\infty o(v^{(i)})\biggr\}.
\]
Lemma \ref{existence} guarantees that $E_d$ is nonempty; in fact, $E_d$ is infinite. For any $u=\sum_{i=1}^\infty u^{(i)}\in E_d,$ if there exists $i_0\geq 2$ such that  $\ u^{(i_0)}\notin \operatorname{Ker}L(1)$, let $$\mathfrak t(u)={\rm min}\{i\mid u^{(i)}\notin \operatorname{Ker}L(1), i\geq2\};$$ otherwise, set $\mathfrak t(u)=\infty$. As a matter of fact, the set \(E_d\) provides a characterization for $d$ being an inner derivation.
\begin{prop}\label{n-prop}
Let $V$ be a vertex operator algebra of CFT type and $d$ a derivation of $V.$  Then $d\in \operatorname{IDer}V$ if and only if there exists $b\in E_d$ such that $\mathfrak t(b)=\infty$.

\end{prop}

\begin{proof}

If $d=o(v)\in \operatorname{IDer}V$  for some $v\in V_1$ (see Proposition \ref{cit-fro-dg}), then $v\in E_d$ and $\mathfrak t(v)=\infty$. Conversely, let $b\in E_d$ satisfy $\mathfrak t(b)=\infty$. For any $u=\sum_{i=1}^\infty u^{(i)}\in E_d$, we compute \begin{eqnarray}\label{zerosum}0\!\!\!&=&\!\![L(-1), [L(1), d]] \quad\quad  \quad\quad (\text{by} \ \eqref{d=comm})\nonumber \\
&=&\!\![L(-1),[L(1),\sum_{i=1}^\infty o(u^{(i)}]]=\sum_{i=1}^\infty [L(-1),[L(1), u^{(i)}_{i-1}]]\nonumber\\
&=&\!\!\sum_{i=2}^\infty [L(-1),(i-1)u^{(i)}_i+(L(1)u^{(i)})_{i-1}]\quad\quad (\text{by \cite[Lemma 2.5]{DLMM}})\\
&=&\!\!-\sum_{i=2}^\infty (i-1)o(iu^{(i)}+L(1)u^{(i)})\nonumber\\
&=&\!\!-\sum_{i=1}^{\mathfrak t(u)-2}i(i-1)o(u^{(i)})-\sum_{i=\mathfrak t(u)-1}^\infty o(iL(1)u^{(i+1)}+i(i-1)u^{(i)}). \nonumber\end{eqnarray} Since $u=\sum_{i=1}^\infty u^{(i)}\in E_d$, we deduce that \begin{equation}\label{gener-zer}\sum_{i=1}^{\mathfrak t(u)-2}i(i-1)u^{(i)}+\sum_{i=\mathfrak t(u)-1}^\infty (iL(1)u^{(i+1)}+i(i-1)u^{(i)})\in E_0.\end{equation} In particular, applying this  to the element $b$ yields
\begin{equation}\label{new} \sum_{i=2}^\infty i(i-1)b^{(i)}\in E_0.\end{equation}It should be noted that the idea for deriving \eqref{zerosum}  comes essentially  from  \cite{DLMM}.

Since $V$ is rational, it follows from \cite[Theorem 2]{DZ} that $V$ is finitely generated. Accordingly, there exists some $N\in\mathbb Z_+$ such that $\bigoplus_{i=0}^N V_i$ generates $V.$ As $b=\sum_{i=1}^\infty b^{(i)}$ is $o$-additive on $V,$ we may pick $t_N\in\mathbb Z_+$ satisfying
\[
o(b^{(j)})\big|_{\bigoplus_{i=0}^N V_i}=0\quad  \text{for\ all}\ j>t_N.
\]In particular, we have \[d|_{\bigoplus_{i=0}^NV_i}=\sum_{j=1}^{t_N}o(b^{(j)})|_{\bigoplus_{i=0}^NV_i}.\]  Suppose  that there exists $j_0\in\{2,3,\ldots,t_N\}$ such that $b^{(j_0)}\neq0$.
Then in view of \eqref{new}, it is harmless to replace $b$ by  $$b-\frac{1}{j_0(j_0-1)}\sum_{i=1}^\infty i(i-1)b^{(i)}=\sum_{i=1}^\infty (1-\frac{i(i-1)}{j_0(j_0-1)})b^{(i)}.$$ But in this case we would have $b^{(j_0)}=0.$ After repeating this process finitely many steps we  may assume that $b^{(j)}=0$ for all $2\le j\le t_N$. Thus, we further obtain  $$d|_{\bigoplus_{i=0}^NV_i}=o(b^{(1)})|_{\bigoplus_{i=0}^NV_i}.$$ Now by Proposition \ref{cit-fro-dg} and Lemma \ref{lem-m} one has $d=o(b^{(1)})$.
\end{proof}

For a vertex operator algebra $V=\bigoplus_{i\in\Z_+} V_i$, we equip each homogeneous component $V_i$ with the Euclidean topology, \(V=\bigoplus_{i=0}^\infty V_i\)  with the direct sum topology, and $\operatorname{End}V$ with the pointwise convergence topology. We now state the following well-known result.

\begin{lemm}\label{con-lem}
Let $V=\bigoplus_{i\in\Z_+} V_i$ be a vertex operator algebra.
\begin{enumerate}
    \item[(1)] For every integer $i\geq 0,$ the linear map $o\colon V_i \to \operatorname{End}V$ is continuous.
    \item[(2)] The linear map $\tilde{o}\colon V_1\big/\big(J(V)\cap V_1\big) \to \operatorname{End}V$ induced by $o\colon V_1 \to \operatorname{End}V$ is a topological embedding (see \eqref{j(v)}).
\end{enumerate}
\end{lemm}

 Our  goal is to prove that $E_d$ contains an element of $V,$ which is formulated as the following theorem.

\begin{theo}\label{mainr}
Let $V$ be a   simple   rational vertex operator algebra of  CFT type. Then ${\rm Der}\, V=\{o(v)\mid v\in V_1\}=\operatorname{IDer}V.$
\end{theo}
\begin{proof}
Take an arbitrary derivation $d\in \operatorname{Der}V.$ By Proposition \ref{cit-fro-dg}, it suffices to prove that $d = o(b)$ for some $b\in V_1$. We first claim that
\[
\sup_{u\in E_d} \mathfrak{t}(u) = \infty.
\]
Suppose, for contradiction, that this supremum is finite. Then there exists a positive integer $t$ such that $t = \max\big\{\mathfrak{t}(u) \,\big|\, u\in E_d\big\}$. Choose an element $b=\sum_{i=1}^{\infty} b^{(i)}\in E_d$ with $\mathfrak{t}(b)=t$. It then follows that $b^{(t)} \notin \operatorname{Ker}L(1)$. From the decomposition \[V_t=\operatorname{Ker}(L(1): V_t\rightarrow V_{t-1})\oplus L(-1)V_{t-1}\quad  \text{(see \cite[Lemma 3.3]{DLMM})},\] we write $b^{(t)}=x^{(t)}+L(-1)y$ with $x^{(t)}\in\operatorname{Ker} L(1)$ and $y\in V_{t-1}$. Applying  \eqref{lift-de} gives
\[
o(b^{(t)})=o(x^{(t)})-(t-1)o(y).
\]
It is then straightforward to verify that $\mathfrak b=\sum_{i=1}^\infty \mathfrak b^{(i)}\in E_d$, where
\[
\mathfrak b^{(i)}=
\begin{cases}
b^{(i)} & \text{if }1\le i\le t-1 \text{ or }i\ge t+1,\\
b^{(t-1)}-(t-1)y & \text{if }i=t-1,\\
x^{(t)} & \text{if }i=t.
\end{cases}
\]
We apply the same reduction procedure to $\sum_{i=1}^{t-1}\mathfrak b^{(i)}$ as was done for $b^{(t)},$ yielding an element $\sum_{i=1}^{t-1}x^{(i)}\in\bigoplus_{i=1}^{t-1}V_i$ satisfying $\sum_{i=2}^{t-1}x^{(i)}\in\operatorname{Ker} L(1)$ and
\[
o\biggl(\sum_{i=1}^{t-1}\mathfrak b^{(i)}\biggr)=o\biggl(\sum_{i=1}^{t-1}x^{(i)}\biggr).
\]
 This leads to $$x=\sum_{i=1}^tx^{(i)}+\sum_{i=t+1}^\infty b^{(i)}\in E_d,$$ for which $\mathfrak t(x)> \mathfrak t(b)$, contradicting the maximality of $\mathfrak t(b)$.

 Since $\sup\limits_{u\in E_d} \mathfrak{t}(u) = \infty$, for every integer $n\geq 1$, there exists an element
$$
b_n = \sum_{i=1}^{\infty} b_n^{(i)} \in E_d
$$
such that $\mathfrak{t}(b_n) \geq n+2$. Define \[\hat b^{(i)}_n=\left\{\begin{array}{lllll} b_n^{(i)} &\mbox{if}\ 1\le i\le \mathfrak t(b_n)-1;\\ 0 &\mbox{if}\ i=\mathfrak t(b_n);\\ -\frac{1}{i-1} L(-1)b_n^{(i-1)}\ (\text{see}\ \ref{lift-de}) &\mbox{if}\ i\geq \mathfrak t(b_n)+1,\end{array}\right.\] and set $\hat b_n=\sum_{i=1}^\infty \hat b_n^{(i)}$. One checks readily that $\hat b_n\in E_d$ and  $\mathfrak t(\hat b_n)=\mathfrak t(b_n)+1$.  Applying \eqref{gener-zer} to $\hat b_n$ and following the argument of Proposition \ref{n-prop}, we can first  make $b_n^{(\mathfrak t(b_n)-1)}=0$ and then $b_n^{(\mathfrak t(b_n)-2)}=0$. Iterating this procedure, we deduce  $b_n^{(\mathfrak (i)}=0$ for all $2\le i\le \mathfrak t(b_n)-1$.  Using this vanishing condition, we produce  an element $\tilde{b}_n \in E_d$ of the form
$$
\tilde{b}_n = \tilde{b}_n^{(1)} + \sum_{i=\mathfrak{t}(b_n)}^{\infty} \tilde{b}_n^{(i)}.
$$
We first note that $\lim\limits_{n\to\infty} \tilde{b}_n^{(i)} = 0$ for all $i\geq 2$. Combining this with Lemma \ref{con-lem}\,(1), we further obtain
$$
\lim_{n\to\infty} o\left(\tilde{b}_n^{(i)}\right) = 0.
$$
Since $\tilde{b}_n \in E_d$, we have
$$
o(\tilde{b}_n^{(1)}) = d - \sum_{i=2}^{\infty} o\left(\tilde{b}_n^{(i)}\right).
$$
Taking limits on both sides, we obtain
\[
d=\lim_{n\to\infty} o\big(\tilde{b}_n^{(1)}\big)=\lim_{n\to\infty} \tilde o\big(\pi(\tilde{b}_n^{(1)})\big),
\]
where $\pi\colon V_1\to V_1/\big(J(V)\cap V_1\big)$ denotes the quotient projection. By Lemma \ref{con-lem}\,(2), the limit $\lim\limits_{n\to\infty} \pi(\tilde{b}_n^{(1)})$ exists. Let $\xi=\lim\limits_{n\to\infty} \pi(\tilde{b}_n^{(1)})$. Since $\pi$ is surjective, there exists some $b\in V_1$ such that $\pi(b)=\xi$. Consequently,
\[
d=\lim_{n\to\infty} \tilde o\big(\pi(\tilde{b}_n^{(1)})\big)=\tilde o(\xi)=\tilde o(\pi(b))=o(b),
\]
as required.\end{proof}

\section*{Acknowledgment}
J. Han is supported by the National Natural Science Foundation of China (No. 12271406).

\end{document}